\documentclass[12pt,oneside]{article}

\usepackage[english]{babel}
\usepackage{graphicx}
\usepackage{amsfonts}
\usepackage{amsmath}
\usepackage{amssymb}
\usepackage{mathrsfs}
\usepackage{amsthm}
\usepackage{latexsym}

\usepackage[all]{xy}
\usepackage{tikz}
\usepackage{tkz-berge}
\usepackage{tkz-arith}
\usetikzlibrary{arrows,shapes,matrix,decorations.pathmorphing,shapes.geometric,petri}
\usepackage{geometry}

\newtheorem{teo}{Theorem}[section]
\newtheorem{lem}[teo]{Lemma}

\newtheorem{cor}[teo]{Corollary}

\theoremstyle{definition}

\DeclareMathOperator{\Z}{\mathbb{Z}}

\DeclareMathOperator{\diam}{\text{diam}}

\DeclareMathOperator{\olr}{\overleftrightarrow{\hspace{.3cm}}}
\DeclareMathOperator{\olra}{\overleftrightarrow{\hspace{.5cm}}}

\DeclareMathOperator{\olrc}{\overleftrightarrow{\hspace{.7cm}}}

\DeclareMathOperator{\la}{\langle}
\DeclareMathOperator{\ra}{\rangle}

\def\Z{{\mathbb{Z}}}

\def\diam{{\hbox{diam}}}

\def\olr{\overleftrightarrow{\hspace{.3cm}}}
\def\olra{\overleftrightarrow{\hspace{.5cm}}}

\def\olrc{\overleftrightarrow{\hspace{.7cm}}}

\def\src{{src^*}}
\def\rcn{rc^*}

\def\la{\langle}
\def\ra{\rangle}

\begin{document}

\centerline{\bf Rainbow connection in some digraphs}
\vskip 1pc
\centerline{Jes\'us Alva-Samos\footnote{
              Instituto de Matem\'aticas, UNAM}  \hskip 1pc
        Juan Jos\'e Montellano-Ballesteros\footnote{Instituto de Matem\'aticas, UNAM. 
                juancho@matem.unam.mx }
         %etc.
}
\date{Received: date / Accepted: date}
% The correct dates will be entered by the editor

\begin{abstract}
An edge-coloured graph $G$ is {\it rainbow connected} if any two vertices are connected by a path whose edges have distinct colours. 
This concept  was introduced by Chartrand et al. in \cite{ch01}, and it was extended to oriented graphs by Dorbec et al. in \cite{DI}. 
 In this paper we present  some results regarding this extention, mostly for the case of circulant digraphs.

\end{abstract}


\section{Introduction}

Given a connected graph $G=(V(G), E(G))$, an edge-coloring of $G$ is called {\it rainbow connected} if for every pair of distinct vertices $u, v$ of  $G$ there is a $uv$-path all whose edges received different colors. The {\it rainbow connectivity  number of} $G$ is the minimum number $rc(G)$ such that there is a rainbow connected  edge-coloring of $G$ with $rc(G)$  colors.  Similarily, an edge-coloring of $G$ is called {\it strong rainbow connected} if for every pair $u, v\in V(G)$ there is a $uv$-path of minimal length (a $uv$-geodesic) all whose edges received different colors. The {\it strong rainbow connectivity  number of} $G$ is the minimum number $src(G)$ such that there is a strong rainbow connected  edge-coloring of $G$ with $src(G)$  colors.

 The concepts of rainbow connectivity  and strong rainbow connectivity of a graph were introduced by Chartrand et al. in \cite{ch01} and, been the connectivity    one fundamental notion in Graph Theory, it is not surprising that several works around these concepts  has been done since then (see for instance \cite{cha,ch02,sc02,sc03,sc01,kri,li3,li4,li2}). For a survey in this topic see (\cite{li1}).     As a natural extension of this notions is that of the {\it rainbow connection} and {\it strong rainbow connection} in oriented graphs, which was introduced by Dorbec et al. in  \cite{DI}. 
 
   Let $D=(V(D), A(D))$ be a strong connected digraph and $\Gamma : A(D)\rightarrow \{1, \dots, k\}$ be an arc-coloring of $D$.  Given $x, y\in V(D)$, a directed $xy$-path $T$ in $D$ will be called {\it rainbow} if no two arcs of $T$ receive the same color.  $\Gamma$ will be called {\it rainbow connected}  if for every pair of vertices $x, y \in V(D)$ there is a rainbow $xy$-path and a rainbow $yx$-path. The {\it rainbow connection number of} $D$, denoted as $\rcn(D)$, is the minimum number $k$ such that there is a rainbow connected arc-coloring  of $D$ with  $k$ colors. Given a pair of vertices $x,y\in V(D)$, an $xy$-path $T$ will be called an {\it $xy$-geodesic}  if the length of $T$ is the distance, $d_D(x,y)$, from $x$ to $y$ in $D$.   An arc-coloring  of $D$ will be called {\it strongly rainbow connected} if for every pair of distinct vertices $x, y$ of $D$ there is a rainbow $xy$-geodesic and a rainbow $yx$-geodesic. The {\it strong rainbow connection number of} $D$, denoted as $\src(D)$, is the minimum number $k$ such that there is a strong rainbow connected arc-coloring  of $D$ with  $k$ colors.
 
In this paper we present some results regarding  this problem, mainly for the case of circulant digraphs.  For general concepts we may refer the reader to  \cite{bang}. 

\section{Some  remarks and  basic results on biorientations of graphs}

Let $D=(V(D), A(D))$ be a strong connected digraph of order $n$ and let $\diam(D)$ be the diameter  of $D$.  As we see in \cite{DI}, it follows that  $$\diam(D)\le \rcn(D)\le \src(D)\leq n.$$ 

%\begin{lem}\label{cotasup} If $D$ is a digraph with $n$ vertices, then $\src(D)\le n$.
%\end{lem}

%\begin{proof}
%Let $v_1,v_2,\dots,v_n$ be the vertices of $D$ and assign the colour $i$ to each arc of the form $v_iv_j$ for any $j$, with $1\le j\le n$. Notice that every path is rainbow, since has no repeated vertices. Thus, each geodesic is also rainbow and the result follows.
%\end{proof}

Also, it is not hard to see that  if  $H$ is a strong spanning subdigraph of $D$, then $\rcn(D)\leq \rcn(H)$.  However, as in the graph case (see\cite{cha}), this is not true for the strong rainbow connection number, as we see in the next lemma. 

\begin{lem}\label{srcGH}
There is a digraph $D$ and a spanning subdigraph $H$ of $D$ such that $\src(D)>\src(H)$.
\end{lem}

\begin{proof}
Let $H$ be as in Figure \ref{fig1}, where $D$ is obtained from $H$ by adding the arc $a_1a_2$.  It is not hard to sse that the colouring in Figure  \ref{fig1} is a strong rainbow 6-coloring of $H$, thus  $\src(H)\leq 6$. We will show that $\src(D)\geq 7$.  Suppose there is a strong rainbow 6-coloring  $\rho$  of $D$, First notice that, for each $i$ and $j$,  the $u_iv_j$-geodesic is unique and contains the arcs $u_iv_i$ and $u_jv_j$, hence there are no two arcs of the type $u_iv_i$ sharing the same colour. Without loss of generality let $\rho(u_iv_i)=i$ for $1\leq i \leq 4$. By an analogous argument,  since $P_i=u_iv_ia_1a_2u_4v_4$ is the only $u_iv_4$-geodesic for $i\le3$, and $a_1a_2, a_2u_4\in A(P_i)$, we can suppose that  such arcs have colours 5 and 6,  respectively. If we assign any of the six colours to the arc $v_1a_1$, we see that for some $j\geq 2$ the unique $u_1v_j$-geodesic is no rainbow, contradicting the choice of $\rho$. Hence $\src(G)\ge7$ and the result follows.
\end{proof}

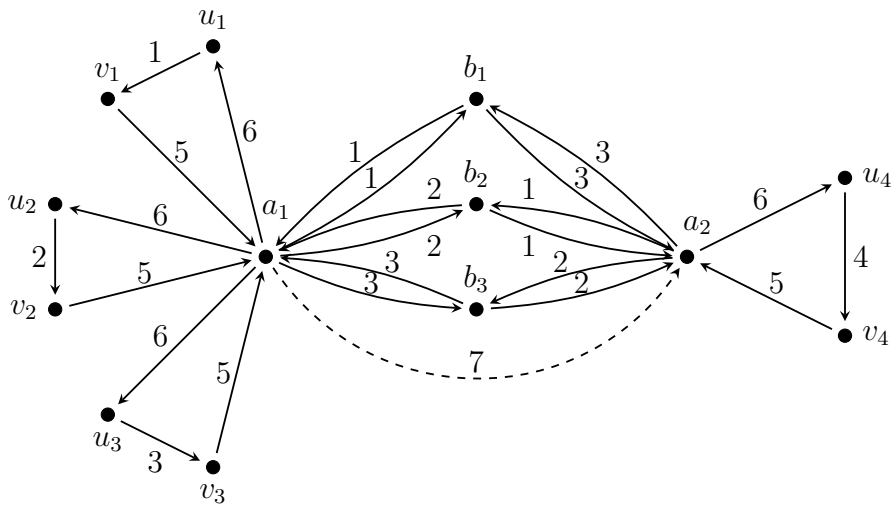
\begin{figure}[ht] %BLACK & WHITE VERSION
\begin{center}
\begin{tikzpicture}[every circle node/.style ={circle,draw,fill=black,minimum size= 5pt,inner sep=0pt, outer sep=0pt},
every rectangle node/.style ={}];

\begin{scope}[scale=0.7, xshift=-5cm]
\node [circle] (0) at (-3,4)[label=90:$u_1$]{};
\node [circle] (1) at (-5,3)[label=90:$v_1$]{};

\node [circle] (2) at (-6,1)[label=180:$u_2$]{};
\node [circle] (3) at (-6,-1)[label=180:$v_2$]{};

\node [circle] (4) at (-5,-3)[label=270:$u_3$]{};
\node [circle] (5) at (-3,-4)[label=270:$v_3$]{};

\node [circle] (6) at (-2,0)[]{};
\node [rectangle] (a1) at (-1.8,0.9) {$a_1$};
\node [circle] (7) at (2,3)[label=90:$b_1$]{};
\node [circle] (8) at (2,1)[label=90:$b_2$]{};
\node [circle] (9) at (2,-1)[label=90:$b_3$]{};
\node [circle] (10) at (6,0)[]{};
\node [rectangle] (a2) at (6.2,0.6) {$a_2$};

\node [circle] (11) at (9,1.5)[label=0:$u_4$]{};
\node [circle] (12) at (9,-1.5)[label=0:$v_4$]{};

\foreach \from/\to in {0/1}
\draw [->, shorten <=3pt, shorten >=3pt, >=stealth, line width=.7pt] (\from) to (\to);
\foreach \from/\to in {2/3}
\draw [->, shorten <=3pt, shorten >=3pt, >=stealth, line width=.7pt] (\from) to (\to);
\foreach \from/\to in {4/5}
\draw [->, shorten <=3pt, shorten >=3pt, >=stealth, line width=.7pt] (\from) to (\to);
\foreach \from/\to in {11/12}
\draw [->, shorten <=3pt, shorten >=3pt, >=stealth, line width=.7pt] (\from) to (\to);
\foreach \from/\to in {1/6,3/6,5/6,12/10}
\draw [->, shorten <=3pt, shorten >=3pt, >=stealth, line width=.7pt] (\from) to (\to);
\foreach \from/\to in {6/0,6/2,6/4,10/11}
\draw [->, shorten <=3pt, shorten >=3pt, >=stealth, line width=.7pt] (\from) to (\to);

\foreach \from/\to in {6/7,7/6,10/8,8/10}
\draw [->, shorten <=3pt, shorten >=3pt, >=stealth, line width=.7pt] (\from) to [bend right=10] (\to);
\foreach \from/\to in {6/8,8/6,10/9,9/10}
\draw [->, shorten <=3pt, shorten >=3pt, >=stealth, line width=.7pt] (\from) to [bend right=10] (\to);
\foreach \from/\to in {6/9,9/6,10/7,7/10}
\draw [->, shorten <=3pt, shorten >=3pt, >=stealth, line width=.7pt] (\from) to [bend right=10] (\to);
%\foreach \from/\to in {6/7,7/6,6/8,8/6,6/9,9/6,10/7,7/10,10/8,8/10,10/9,9/10}
%\draw [->, shorten <=3pt, shorten >=3pt, >=stealth, line width=.7pt, black] (\from) to [bend right=10] (\to);

\draw[->, shorten <=5pt, shorten >=5pt, >=stealth, line width=.7pt, black, dashed] (-2,0) .. controls (0,-3) and (4,-3) .. (6, 0);

\node [rectangle] () at (-4.1,3.9) {1}; \node [rectangle] () at (-0.3,2) {1}; \node [rectangle] () at (0,1.5) {1};
\node [rectangle] () at (3,0.2) {1}; \node [rectangle] () at (3,1.3) {1};
\node [rectangle] () at (-6.3,0) {2}; \node [rectangle] () at (1.2,0.2) {2}; \node [rectangle] () at (1.2,1.3) {2};
\node [rectangle] () at (3.6,-0.1) {2}; \node [rectangle] () at (4,-0.5) {2};
\node [rectangle] () at (-4.1,-3.9) {3}; \node [rectangle] () at (0.4,-0.1) {3}; \node [rectangle] () at (0,-0.5) {3};
\node [rectangle] () at (4,1.5) {3}; \node [rectangle] () at (4.4,2) {3};
\node [rectangle] () at (9.3,0) {4};
\node [rectangle] () at (-3.6,2) {5}; \node [rectangle] () at (-4.3,-0.3) {5}; \node [rectangle] () at (-2.8,-2.2) {5};
\node [rectangle] () at (7.7,-0.5) {5}; 
\node [rectangle] () at (-2.3,2.4) {6}; \node [rectangle] () at (-4,0.8) {6}; \node [rectangle] () at (-4,-1.5) {6};
\node [rectangle] () at (7.4,1.1) {6}; 
\node [rectangle] () at (2,-2) {7};

\end{scope}

\end{tikzpicture}
\end{center}\caption{The digraphs $D$ and $H$ from Lemma \ref{srcGH}.}\label{fig1}
\end{figure}

%\begin{proof} Let $H$ be as in Figure \ref{fig1}, where $G$ is obtained from $H$ by adding the arc $(a_1,a_2)$.  It is not hard to sse that the colouring in Figure  \ref{fig:srcGH} is a strong rainbow 6-coloring of $H$, thus  $\src(H)\leq 6$. We will show that $\src(G)\geq 7$.  Suppose there is a strong rainbow 6-coloring  $\rho$  of $G$, First notice that, for each $i$ and $j$,  the $u_iv_j$-geodesic is unique and contains the arcs $(u_i,v_i)$ and $(u_j,v_j)$, hence there are no two arcs of the type $(u_i,v_i)$ sharing the same colour. Without loss of generality let $\rho(u_i,v_i)=i$ for $1\leq i \leq 4$. By the same argument, $P_i=u_iv_ia_1a_2u_4v_4$ is the only $u_iv_4$-geodesic for $i\le3$, and $(a_1,a_2),(a_2,u_4)\in A(P_i)$, then such arcs have colours 5 and 6 respectively. If we assign any of the six colours to the arc $(v_1,a_1)$, then $u_1v_j$ is no rainbow for some $j\ge2$, contradicting the choice of $\rho$. Hence $\src(G)\ge7$ and the result follows. \end{proof}

%\centerline{\bf FIGURA}

Given a pair $v, u\in V(D)$,  if the arcs $uv$ and $vu$ are in $D$, then we say that $uv$ and $vu$ are \emph{symmetric} arcs. When every arc of $D$ is symmetric, $D$ is called a \emph{symmetric} digraph. 
Given a graph $G=(V(G), E(G))$, its {\it biorientation} is the symmetric digraph   $\stackrel{\olr}{G}$ obtained from $G$ by replacing 
each edge $uv$ of $G$ by the pair of symmetric arcs $uv$ and $vu$. 

Given a graph $G$ and  a (strong) rainbow connected edge-coloring of  $G$, it is not hard to see that the arc-coloring  of   $\stackrel{\olr}{G}$, obtained by assign the color of the edge $uv$ to both arcs $uv$ and $vu$  is a (strong) rainbow connected arc-coloring of  $\stackrel{\olr}{G}$. Thus $\rcn(\stackrel{\olr}{G})\le rc(G)$ and $\src(\stackrel{\olr}{G})\le src(G)$. Although for some graphs and its biorientations these values coincide (for instance, as we will see, for $n\geq 4$,   $rc(C_n) = src(C_n)= \rcn(\stackrel{\olra}{C_n})=\src(\stackrel{\olra}{C_n})$),   for other graphs and its biorientations the difference between those values is unbounded, as we see in the case of the stars, where for each $n\geq 2$, $rc(K_{1,n})=n$ (for each path between terminal vertices we need two colors)  and $\rcn(\stackrel{\olrc}{K_{1,n}})=\src(\stackrel{\olrc}{K_{1,n}})=2$ (the colouring that assigns color 1 to the in-arcs of the ``central''
  vertex and assigns color 2 to the ex-arcs of the central vertex is a strong rainbow coloring).

\begin{teo}\label{srcKn}
Let $D$ be a nontrivial digraph, then
\begin{enumerate}
	\item[(a)] $\src(D)=1$ if and only if $\rcn(D)=1$ if and only if, for some $n\geq 2$,  $D= \stackrel{\olr}{K_n}	$;
	\item[(b)] $\rcn(D)=2$ if and only if $\src(D)=2$.
\end{enumerate}
\end{teo}
\begin{proof} First observe that since $D$ is nontrivial, $\rcn(D)\geq 1$ and therefore if $\src(D)=1$ then $\rcn(D)=1$. If $\rcn(D)= 1$ then $\diam(D) =1$ and hence $D=\stackrel{\olr}{K_n}$ for some $n\geq 2$.  On the other hand, if $D= \stackrel{\olr}{K_n}$ it follows that every 1-colouring of $D$ is a strong rainbow colouring. Thus $1\geq \src(D)\geq \rcn(D)\geq 1$ and (a) follows.  For (b), if $\src(D) = 2$, by (a), $\rcn(D)>1$ and hence $\rcn(D) = 2$.  If $\rcn(D) = 2$,   $D$ has a 2-rainbow colouring and, by (a) $D\not=\stackrel{\olr}{K_n}$. Therefore for every pair $u, v\in V(D)$, with $d(u,v)\geq 2$,  exists a $uv$-rainbow path of lenght 2, which is also geodesic. Hence $\src(D)=2$ and (b) folows. \end{proof}

\begin{teo}\label{srcCn} 
\begin{enumerate} 
	\item[(a)] For $n\geq 2$,  $\rcn(\stackrel{\olra}{P_n})=\src(\stackrel{\olra}{P_n})=n-1$; 
	\item[(b)] For $n\geq 4$, $\rcn(\stackrel{\olra}{C_n})=\src(\stackrel{\olra}{C_n})=\lceil n/2\rceil$
\item[(c)]  Let $k\ge2$, if $\stackrel{\olr}{K}_{n_1,n_2,\dots,n_k}$ is the complete $k$-partite digraph where $n_i\ge2$ for some $i$, then $\rcn(\stackrel{\olr}{K}_{n_1,n_2,\dots,n_k})=\src(\stackrel{\olr}{K}_{n_1,n_2,\dots,n_k})=2$.
\end{enumerate}
\end{teo}

\begin{proof}  In  \cite{ch01} it is shown that for every $n\geq 4$,  $src(C_n)=\lceil \frac{n}{2}\rceil$ and for every $n\geq 2$, $src(P_n) = n-1$.  Since $\diam(\stackrel{\olra}{P_n})= n-1$ it follows that $n-1\leq \rcn(\stackrel{\olra}{P_n})\leq \src(\stackrel{\olra}{P_n})\leq src(P_n) = n-1$ and the first part of the theorem follows. In an analogous way, if $n$ is even, $\lceil \frac{n}{2}\rceil=\diam(\stackrel{\olra}{C_n})\leq \rcn(\stackrel{\olra}{C_n})$ and since $\src(\stackrel{\olra}{C_n})\le src(C_n)=\lceil \frac{n}{2}\rceil$, $\rcn(\stackrel{\olra}{C_n})=\src(\stackrel{\olra}{C_n})=\lceil \frac{n}{2}\rceil$.  Let  $n =2k+1$ with $k\ge2$ and let us suppose there is  a rainbow $k$-colouring  $\rho$ of $\stackrel{\olra}{C_n}$.  Observe that for every $0\leq i \leq n-1$, $(v_i, v_{i+1}, \dots v_{i+k})$ is the only $v_iv_{i+k}$-path of length  $d(v_i, v_{i+k}) = k$ in $\stackrel{\olra}{C_n}$ and therefore  the $k$ colours of $\rho$ occurs in each of such geodesic paths. Thus $\rho(v_iv_{i+1}) = \rho(v_{i+k}v_{i+k+1})$  for each $0\leq i \leq n-1$, which, since $(k, n=2k+1)=1$ implies that all the arcs $v_iv_{i+1}$ in $\stackrel{\olra}{C_n}$ receive the same color which is a contradiction.  Thus $\rcn(\stackrel{\olra}{C_n}) \geq k +1 = \lceil \frac{n}{2}\rceil$ and (b) follows. For (c), since  $n_i\ge2$ for some $i$, then $\stackrel{\olr}{K}_{n_1,n_2,\dots,n_k}$ is not a complete digraph, hence $\rcn(\stackrel{\olr}{K}_{n_1,n_2,\dots,n_k})\ge2$. Let $V_1,V_2,\dots,V_k$ be the $k$-partition on independent sets of $V(\stackrel{\olr}{K}_{n_1,n_2,\dots,n_k})$, and for each arc $uv$, with $u\in V_i$ and $v\in V_j$, assign color 1 to $uv$ if $i<j$ and color 2 if $i>j$. Since $\diam(\stackrel{\olr}{K}_{n_1,n_2,\dots,n_k})=2$, it is not hard to see that this is a strong rainbow connected 2-coloring and therefore $\src(\stackrel{\olr}{K}_{n_1,n_2,\dots,n_k})\le2$. 
\end{proof}

\begin{teo}\label{teo3} Let $D$ be a spanning strong connected subdigraph of $\stackrel{\olra}{C_n}$ with $k \geq 1$ asymmetric arcs.  Thus
$$\rcn(D)=\left\{\begin{array}{ll} n-1 & \text{if }k\leq2\text{;} \\ n & \text{if }k\geq3\text{.}\end{array}\right.$$
Moreover, if $k\geq 3$, $\rcn(D)= \src(D) = n$. 
\end{teo}
\begin{proof}  Let $V(\stackrel{\olra}{C_n}) = \{v_0, \dots , v_{n-1}\}$ and  suppose $v_0v_{n-1}\not\in A(D)$. Since $D$ is strong connected the $v_0v_{n-1}$-path $T= (v_0, v_1, \dots , v_{n-1})$ is contained in $D$, thus $\diam(D)\geq n-1$. Therefore,  $n-1\leq \rcn(D)\leq n$.   If $k=1$ we see that $\stackrel{\olra}{P_n}$ is a spanning subdigraph of $D$, hence $n-1\leq \rcn(D)\leq \rcn(\stackrel{\olra}{P_n})$, which by Theorem \ref{srcCn} (a) implies that $\rcn(D)=n-1$.  Let $k\geq 2$.  If $v_{n-1}v_0\not\in A(D)$, since $D$ is strong connected it follows that $D$ is isomorphic to $\stackrel{\olra}{P_n}$  which have no asymmetric arcs and thus this is not possible.  Therefore  $v_{n-1}v_0\in A(D)$.   If  there is a $(n-1)$-rainbow coloring $\rho$ of $D$,  since  $v_{n-1}v_0\in A(D)$,  the directed cycle $C$ induced by $A(T)\cup v_{n-1}v_0$ is a spanning subdigraph of $D$ and therefore there are two arcs $v_iv_{i+1}, v_jv_{j+1}\in A(C)$ such that $\rho(v_iv_{i+1}) = \rho(v_jv_{j+1})$. Since $\rho$ is a rainbow coloring, there is a rainbow $v_iv_{j+1}$-path and a rainbow $v_jv_{i+1}$-path in $D$. Thus the paths $(v_i, v_{i-1}, \dots, v_{j+2}, v_{j+1})$ and $(v_j, v_{j-1}, \dots, v_{i+2}, v_{i+1})$ most be contained in $D$ and therefore the number of assymetric arcs in $D$ is at most 2. Thus, if $k\geq 3$ then $\rcn(D)\geq n$ and hence, $\rcn(D) = n$. Finally,  if $k=2$, let $\rho$ be the $(n-1)$-arc coloring of $D$ which assigns the same color to the assymetric arcs, and for the remaining $n-2$ pairs of simmetric arcs and the remaining $n-2$ colors,  $\rho$ assigns the same color to each pair of simmetric arcs. It is not hard to see that $\rho$ is a rainbow coloring of $D$, thus $\rcn(D)\leq n-1$ and the first part of the theorem follows. The second part is directly from the first part of the theorem and from the fact that $\src(D)\leq n$.  \end{proof}

As a direct corollary of the previous result we have

\begin{cor}  Let $D$ be a strong connected digraph with $m\geq 3$ arcs. Thus $\rcn(D)= \src(D)=m$ if and only if $D = \stackrel{\rightarrow}{C_m} $.\end{cor}

\section{Circulant digraphs}

For an integer $n\geq2$ and a set $S\subseteq\{1, 2,\dots, n-1\}$, the \emph{circulant digraph} $C_n(S)$ is defined as follows: $V(C_n(S))=\{v_0,v_1,\dots,v_{n-1}\}$ and 
$$A(C_n(S))=\{v_iv_j  :  j-i\stackrel{n}{\equiv}s,\ s\in S\},$$ where $a\stackrel{n}{\equiv}b$ means:  {\it $a$  congruent with $b$ modulo $n$}. 
The elements of $S$ are called \emph{generators}, and an arrow $v_iv_j$,  where $j-i\stackrel{n}{\equiv}s\in S$,  will be called an {\it $s$-jump}. If $s\in S$ we denote by $C_{(s)}$ the spanning subdigraph of $C_n(S)$ induced by all the $s$-jumps. Observe that for every pair of vertices $v_i$ and $v_j$ there is at most one $v_iv_j$-path in $C_{(s)}$. If such $v_iv_j$-path in $C_{(s)}$ exists will be denoted by $v_iC_{(s)}v_j$. From now on the subscripts of the vertices are taken modulo $n$.
Given an integer $k\geq 1$, let  $[k]=\{1,2,\dots,k\}$. 

\begin{teo}\label{Cn[k]}
If $1\leq k\leq n-2$, then $\rcn(C_n([k]))=\src(C_n([k]))=\lceil \frac{n}{k}\rceil$.
\end{teo}
\begin{proof} Let $D = C_n[k]$. The case when $k=1$ is proved in Theorem \ref{teo3}. Let  $2\leq k\leq n-2$, and $V(D) = \{v_0, \dots ,v_{n-1}\}$.  By definition it follows that for every pair $0\leq i \le j\leq n-1$, $d(v_i, v_j) = d(v_0, v_{j-i})$ and $d(v_j,v_i) = d(v_0, v_{i+n-j})$. 
Also it is not hard to see that for every $0\leq i\leq n-1$, $d(v_0, v_i)= \lceil\frac{i}{k}\rceil$. From here it follows that $\diam(D)= \lceil\frac{n-1}{k}\rceil$.  

Let $P=\{V_1, V_2, \dots, V_{\lceil\frac{n}{k}\rceil}\}$ be a partition of $V(D)$ such that for each $i$, with $1\leq i \leq \lfloor\frac{n}{k}\rfloor$, $V_i=\{v_j :(i-1)k\leq j\leq ik-1\}$ and, if $\lceil\frac{n}{k}\rceil \not= \lfloor\frac{n}{k}\rfloor$, $V_{\lceil\frac{n}{k}\rceil} =\{\ v_j : k\lfloor\frac{n}{k}\rfloor\leq j \leq n-1 \}$.

\noindent {\bf Claim 1} For every pair $v_i, v_j\in V(D)$ there is a $v_iv_j$-geodesic path $T$ such that for every $V_p\in P$, $|V_p\cap V(T\setminus v_j)|\leq 1$.

\noindent Let $v_{rk+i}, v_{sk+j}\in V(D)$. If  $r\not =s$ let $0\leq q\leq k-1$ and $t$  be the minimum integer such that $(r+t)k + i +q \stackrel{n}{\equiv} sk+j$  and let $$T=(v_{rk+i}, v_{(r+1)k+i}, \dots,  v_{(r+t)k+i},  v_{(r+t)k+i+q})$$  be a $v_{rk+i}v_{sk+j}$-path.  Since $t$ is minimum and $0\leq q\leq k-1$ it follows that $T$ is a   $v_{rk+i}v_{sk+j}$-geodesic path and, since for every $V_p\in P$, $|V_p|\leq k$, hence for every $V_p\in P$, $|V_p\cap V(T\setminus v_{sk+j})|\leq 1$. 

 If $r=s$ and $i\leq j$ it follows that $v_{rk+i}v_{sk+j}\in A(D)$ and $T=(v_{rk+i}, v_{sk+j})$ is a $v_{rk+i}v_{sk+j}$-geodesic path with the desired properties. So, let us suppose  $i\geq j+1$. Thus $$d(v_{rk+i}, v_{sk+j})=\lceil\frac{n-k(r-s)-(i-j)}{k}\rceil = \lceil\frac{n-(i-j)}{k}\rceil.$$  Let $t$ be the maximum integer such that $(r+t)k+i\leq n-1$.  If $v_{(r+t)k+i}v_j\in A(D)$, then $$T=(v_{rk+i}, v_{(r+1)k+i}, \dots, v_{(r+t)k+i}, v_j, v_{k+j}, \dots v_{sk+j})$$ is a $v_{rk+i}v_{sk+j}$-geodesic path such that for every $V_p\in P$, $|V_p\cap V(T\setminus v_{sk+j})|\leq 1$. 
If $v_{(r+t)k+i}v_j\not\in A(D)$, since $i\geq j+1$ and $t$ is maximum, it follows that $v_{(r+t)k + i}\in V_{\lceil\frac{n}{k}\rceil-1}$ and $v_{(r+t)k + i}v_{n-1}\in A(D)$. Therefore $$T=(v_{rk+i}, \dots, v_{(r+t)k+i}, v_{n-1}, v_j, v_{k+j}, \dots v_{sk+j})$$ 
is a $v_{rk+i}v_{sk+j}$-geodesic path such that for every $V_p\in P$, $|V_p\cap V(T\setminus v_{sk+j})|\leq 1$, and the claim follows.\\
%%%%%%%%%%%%%%

Let $\rho:A(D)\rightarrow \{1, 2, \dots, \lceil\frac{n}{k}\rceil\}$ be the arc-coloring of $D$
defined as follows: for every $v_iv_j\in A(D)$, $\rho(v_iv_j)= p$ if and only if $i\in V_p$. Given $v_i, v_j\in V(D)$, from Claim 1  we see there is a  $v_iv_j$-geodesic path $T$ such that for every $V_i\in P$, $|V_i\cap V(T\setminus v_j)|\leq 1$ which, by definition of $\rho$, is a rainbow path. From here it follows that $\rho$ is a strong rainbow coloring of $D$. Thus, $\src(D)\leq \lceil\frac{n}{k}\rceil$, and since $\diam(D)= \lceil\frac{n-1}{k}\rceil$, for every $n$ such that $\lceil\frac{n}{k}\rceil = \lceil\frac{n-1}{k}\rceil$
 we have $\rcn(D) = \src(D) = \lceil\frac{n}{k}\rceil$. Hence, to end the proof just remain to verify the case $n=kt+1$. Let suppose there is a $t$-rainbow coloring $\rho$ of $D$, and consider $C_{(k)}$,  the spanning subdigraph of $D$ induced by the $k$-jumps.  Since $(k, n=kt+1)=1$ it follows that $C_{(k)}$ is a cycle, and each $v_iv_{i+tk}$-path in $C_{(k)}$ is the only $v_iv_{i+tk}$-path of length $t$ in $D$. Thus, since $\rho$ is a $t$-rainbow coloring, in every $v_iv_{i+tk}$-path in $C_{(k)}$ most appear the $t$ colors. Therefore, for every $0\leq i\leq n-1$, $\rho(v_iv_{i+k}) = \rho(v_{i+kt}v_{i+k(t+1)})$, which, since $(k, n=kt+1)=1$, implies that every arc in $C_{(k)}$ receives the same color which is a contradiction. Therefore $\rcn(D)\geq t+1 = \lceil\frac{n}{k}\rceil$ and since $\src(D)\leq \lceil\frac{n}{k}\rceil$, the theorem follows. \end{proof}

Now, we turn our attention on the circulant digraphs with a pair of generators $\{1, k\}$, with $2\leq k\leq n-1$. Observe that  for every circulant digraph $C_n(\{a_1,a_2\})$, if $(a_1, n )= 1$ and $b\in \Z_n$ is the solution of $a_1x\stackrel{n}{\equiv}1$, then $C_n(\{1,ba_2\}))\cong C_n(\{a_1,a_2\}))$. 
From here, we obtain the following.

\begin{cor}\label{srcC.n.1,2}
For  $k\ge1$, $\rcn(C_{2k+1}(1,k+1))=\src(C_{2k+1}(1,k+1))=k+1$.
\end{cor}
\begin{proof}  By Theorem \ref{Cn[k]}, for every $n\geq 4$,  $\rcn(C_{n}([2])=\src(C_{n}([2]))= \lceil\frac{n}{2}\rceil$. Since $(k+1, 2k+1) =1$ and 2  is the solution of $(k+1)x\stackrel{2k+1}{\equiv}1$, then $C_{2k+1}(\{1,k+1\}))\cong C_{2k+1}(\{1, 2\})) = C_{2k+1}([2])$ and the result follows.\end{proof}

 Observe that given any circulant digraph $C_n(\{1,k\})$, for every pair $v_i, v_j \in C_n(\{1,k\})$ we have $d(v_i,v_j)=d(v_0,v_{j-i})$ (where $j-i$ is taken modulo $n$). Thus, $\diam(C_n(\{1,k\}))= max\{ d(v_0, v_i) : v_i\in V(C_{n}(\{1,k\}))\}$.

Given two positive integers $i, k$, let denote as $re(i, k)$ the residue of  $i$ modulo $k$.

\begin{lem}\label{diametro} Let $C_n(\{1,k\})$ be a circulant digraph and $V= \{v_0, \dots, v_{n-1}\}$ it set of vertices .  If  $n\geq (k-1)\lceil\frac{n}{k}\rceil$ then for every $v_i\in V$, $d(v_0, v_i) = \lfloor\frac{i}{k}\rfloor + re(i, k)$. 

Moreover $\diam(C_n(\{1,k\})) = \lfloor\frac{n-1}{k}\rfloor +max\{    re(n-1, k),  k-2\}$.
\end{lem} 
\begin{proof}  Let $v_i\in V$, $P=(v_0=u_0, u_1\dots, u_s=v_i)$ be a $v_0v_i$-geodesic path with a minimum number of $k$-jumps,  and suppose in $P$ there are  $p$ $k$-jumps and $q$ $1$-jumps.  Also suppose the first $p$ steps of $P$ are $k$-jumps, and the last $q$ are $1$-jumps. Thus $d(v_0, v_i)= p+q$. Since $P$ is geodesic, it follows that $q\leq k-1$ and therefore $p\geq \lfloor\frac{i}{k}\rfloor$. Hence $v_{k\lfloor\frac{i}{k}\rfloor}=u_{\lfloor\frac{i}{k}\rfloor}\in V(P)$ and the subpath $$Q=(v_{k\lfloor\frac{i}{k}\rfloor}=u_{\lfloor\frac{i}{k}\rfloor}\dots, u_j,\dots,u_s=v_i)$$ is a $v_{k\lfloor\frac{i}{k}\rfloor}v_i$-geodesic path with $p'= p-\lfloor\frac{i}{k}\rfloor$ $k$-jumps and $d(v_{k\lfloor\frac{i}{k}\rfloor}, v_i)= p'+q\leq i-k\lfloor\frac{i}{k}\rfloor= re(i, k)$. If $p>\lfloor\frac{i}{k}\rfloor$ then $q< re(i,k)$ and since $re(i, k) < k$, it follows that $p'\geq \lceil\frac{n}{k}\rceil$. Therefore, if $m= k\lceil\frac{n}{k}\rceil-n$, $v_{k\lfloor\frac{i}{k}\rfloor+m}=u_{\lfloor\frac{i}{k}\rfloor+\lceil\frac{n}{k}\rceil}\in V(Q)$ and the subpath 
$$(v_{k\lfloor\frac{i}{k}\rfloor}=u_{\lfloor\frac{i}{k}\rfloor}\dots, u_j,\dots,u_{\lfloor\frac{i}{k}\rfloor+\lceil\frac{n}{k}\rceil}=v_{k\lfloor\frac{i}{k}\rfloor+m})$$ is a $v_{k\lfloor\frac{i}{k}\rfloor}v_{k\lfloor\frac{i}{k}\rfloor+m}$-geodesic path of $\lceil\frac{n}{k}\rceil$ $k$-jumps and $d(v_{k\lfloor\frac{i}{k}\rfloor}, v_{k\lfloor\frac{i}{k}\rfloor+m})= \lceil\frac{n}{k}\rceil\leq m$. Since $n\geq (k-1)\lceil\frac{n}{k}\rceil$ it follows that $\lceil\frac{n}{k}\rceil\geq k\lceil\frac{n}{k}\rceil-n=m$ and therefore $d(v_{k\lfloor\frac{i}{k}\rfloor}, v_{k\lfloor\frac{i}{k}\rfloor+m})= m$. Thus, replacing in $P$ the subpath  $$(v_{k\lfloor\frac{i}{k}\rfloor}=u_{\lfloor\frac{i}{k}\rfloor}\dots, u_j,\dots,u_{\lfloor\frac{i}{k}\rfloor+\lceil\frac{n}{k}\rceil}=v_{k\lfloor\frac{i}{k}\rfloor+m})$$ by the subpath $$(v_{k\lfloor\frac{i}{k}\rfloor},v_{k\lfloor\frac{i}{k}\rfloor+1}, \dots,v_{k\lfloor\frac{i}{k}\rfloor+m})$$ we obtain a $v_0v_i$-geodesic path with less $k$-jumps than $P$, which is a contradiction. Thus $p=\lfloor\frac{i}{k}\rfloor$ and therefore $q=re(i,k)$ which implies that $d(v_0, v_i) =  \lfloor\frac{i}{k}\rfloor + re(i, k)$ and the first part of the result follows.  For the second part, first observe that $d(v_0, v_{n-1})=\lfloor\frac{n-1}{k}\rfloor + re(n-1, k)$ and $d(v_0, v_{(k\lfloor\frac{n-1}{k}\rfloor) -1})=\lfloor\frac{n-1}{k}\rfloor + k-2 $, thus $\diam(C_n(\{1,k\})) \geq \lfloor\frac{n-1}{k}\rfloor +max\{    re(n-1, k),  k-2\}$. If there is $v_i\in  V$ such that $d(v_0, v_i)> \lfloor\frac{n-1}{k}\rfloor + k-2$, it follows that $n-1\geq i\geq k\lfloor\frac{n-1}{k}\rfloor $ but then $d(v_0, v_i)\leq d(v_0, v_{n-1})=\lfloor\frac{n-1}{k}\rfloor + re(n-1, k)$ and the result follows.
\end{proof}

\begin{teo}\label{srcC.2k.1,k,k+1}
For every integer $k\ge 2$\\
\hspace*{1.15cm}$(i)$ $\rcn(C_{2k}(\{1,k\}))=\src(C_{2k}(\{1,k\}))=k$.\\
\hspace*{1cm}$(ii)$ $\rcn(C_{2k}(\{1,k+1\}))=\src(C_{2k}(\{1,k+1\}))=k$.
\end{teo}

\begin{proof}   Let $V=\{v_0, \dots, v_{2k-1}\}$ be the set of vertices of $C_{2k}(\{1,k\}$. By Lemma \ref{diametro} we see that  $k=\diam(C_{2k}(\{1,k\}))$ and therefore $k\leq \rcn(C_{2k}(\{1,k\}))$.  Let $\{V_0,\dots,V_{k-1}\}$ be a partition of $V$, where $V_r=\{v_r,v_{r+k}\}$ for $0\le r\le k-1$ and define a $k$-colouring $\rho$ such that for every $0\leq r\leq k-1$, $(u,u')\in\rho^{-1}(r)$ if $u\in V_r$.   Let $v_i, v_j\in V$  and suppose $i+q+pk \stackrel{n}{\equiv} j$ where  $d(v_i, v_j) = p+q$ and $q\leq k-1$.  Observe that, since $q<k$,  $v_iC_{(1)}v_{i+q}C_{(k)}v_{i+pk+q}$ is a rainbow $v_iv_j$-path and by Lemma \ref{diametro} is  $v_iv_j$-geodesic. Therefore $\src(C_{2k}(\{1,k\}))\le k$  and (i) follows.  For (ii), let $V =\{v_0, \dots, v_{2k-1}\}$ be the set of vertices of $C_{2k}(\{1,k+1\})$  and let $\{V_0,\dots,V_{k-1}\}$ as before.  By Lemma  \ref{diametro} it follows that $\diam(C_{2k}(\{1,k+1\}))=k$ which implies $k\leq \rcn(C_{2k}(\{1,k+1\}))$.  Now let $\rho$ be a $k$-colouring such that $(u,u')\in\rho^{-1}(r)$ if $u\in V_r$ .  Since $N^+(u)=V_{r+1}$ for each $u\in V_r$ (taken $r+1$  modulo $k$), it follows that  every path of length at most $k$ is rainbow, in particular every geodesic path is rainbow. Thus  $k\geq \src(C_{2k}(\{1,k+1\}))$ and (ii) follows.
\end{proof}

\begin{teo}
For every integer  $k\ge3$ we have $$\src(C_{(k-1)^2}(\{1,k\}))=\rcn(C_{(k-1)^2}(\{1,k\}))=2k-4.$$
\end{teo}

\begin{proof}  By Lemma \ref{diametro} we see that  $\diam(C_{(k-1)^2}(\{1,k\}))=2k-4$ and therefore $\rcn(C_{(k-1)^2}(\{1,k\}))\geq 2k-4$.  Let $V =\{v_0, \dots , v_{(k-1)^2-1}\}$ be the set of vertices of $C_{(k-1)^2}(\{1,k\})$ and for each $i$, with  $0\leq i < (k-1)^2 $,  identify the vertex $v_i$ with the pair  $\la \lfloor\frac{i}{k-1}\rfloor, re(i, k-1)\ra$.  
Let $\mathcal{V}= \{V_0,\dots,V_{k-2}\}$ be a partition of $V$, where $V_r=\{\la r,s\ra\mid 0\le s\le k-2\}$ for $0\le r\le k-2$, and let $\rho$ be a  $(2k-4)$-colouring defined as follows: For each $r$ with $0\le r\leq k-1$, 
\begin{enumerate}
\item  The arc $(\la r,s\ra\la r+1, s\ra)$, with $0\leq s\leq k-2$, receives color $r$. 
\item  The arcs $(\la r,0\ra\la r,1\ra)$ and $(\la r,k-2\ra\la r+1,0\ra)$ receive colour $r$.
\item The arc $(\la r,s\ra\la r,s+1\ra)$,   with $1\leq s\leq k-3$, receives  colour $k-2+s$.
\end{enumerate}
Observe that every path with length at most $k-1$ in $C_{(k)}$ is  rainbow, and, except for those paths of lenght $k-1$ in $C_{(1)}$ starting at $\la r, 0\ra$ (with  $0\leq r<k-1$),  every path in $C_{(1)}$ with lenght at most $k-1$ is rainbow.  From the structure of $\rho$ we see that to prove $\rho$ is a strong coloring we just need to show that for every $v\in V_0$ and every $w\in V$ there is a rainbow $vw$-geodesic path. 

Let $\la 0, s_0\ra \in V_0$ and $\la r, s\ra\in V_r$. Since $\la 0, s_0\ra = v_{s_0}$ and $\la r, s\ra = v_{r(k-1)+s}$, by Lemma \ref{diametro}, $$d(v_{s_0}, v_{r(k-1)+s}) = \lfloor \frac{(k-1)r+ s-s_0}{k}\rfloor + re((k-1)r+s-s_0, k)$$ (taken $(k-1)r+s -s_0$ modulo $(k-1)^2$).  Thus, if $t= \lfloor \frac{(k-1)r+ s-s_0}{k}\rfloor$, $$P = \la 0, s_0\ra C_{(k)}\la \lfloor\frac{s_0+tk}{k-1}\rfloor , re(s_0+tk, k-1)\ra C_{(1)} \la r, s \ra$$ is a geodesic path. The subpath in $C_{(k)}$ receives colors $j$, with $0\leq j \leq \lfloor \frac{s_0+tk}{k-1}\rfloor -1\leq k-2$, and the subpath in $C_{(1)}$ receives colors $i$. with $k-1\leq i\leq 2k-3$ or $i=\lfloor \frac{s_0+tk}{k-1}\rfloor $. Thus, if  $P$ is not rainbow then we have that: the subpath in $C_{(1)}$ most be of lenght $k-1$; $\la\lfloor\frac{s_0+tk}{k-1}\rfloor , re(s_0+tk, k-1)\ra = \la r-1, 0\ra$ and $\la r, s\ra = \la r, 0\ra$. 

If $r-1= 0$ it follows that $\la 0, s_0\ra = \la 0, 0\ra$ and the path  $Q$ of $k$-jumps $\la 0, 0\ra C_{(k)}\la 1, 0\ra$ of lenght $k-1$ is a geodesic rainbow. If $r-1=1$, $(\la 0, s_0\ra\la 1, 0\ra)$ most be a $k$-jump which is not possible. If $r-1\geq 2$, let $Q$ be the rainbow geodesic obtained by the concatenation of the paths $\la 0, s_0\ra C_{(k)}\la r-3, k-2\ra$ (which receives colors between $0$ and $r-4$); the arcs $(\la r-3, k-2\ra, \la r-2, 0\ra)$  and $(\la r-2, 0\ra, \la r-1, 1\ra)$ (with colors $r-3$ and $r-2$ respectively);  and $\la r-1, 1\ra C_{(1)}\la r, 0\ra$ (which receives the colors $r-1$ and $k-1, \dots, 2k-3$). Hence, $P$  or $Q$ is a $\la 0, s_0\ra\la r, s\ra$-geodesic rainbow, and the theorem follows.\end{proof}

\begin{teo}\label{src.n=ak.rest}
If $n=a_nk$ with $a_n\ge k-1\ge2$, then $$\src(C_n(\{1,k\}))=\rcn(C_n(\{1,k\}))=a_n+k-2.$$
\end{teo}

\begin{proof}
By Lemma \ref{diametro}  we see that $\diam(C_n(\{1,k\})) = a_n + k-2$ and then to prove the result just remain  to show that $\src(C_n(\{1,k\}))\leq a_n+k-2$. Let $V = \{v_0, \dots, v_{n-1}\}$ be the set of vertices of $C_n(\{1,k\})$ and,  for each $i$, with $0\leq i < n $,  identify the vertex $v_i$ with the pair  $\la \lfloor\frac{i}{k}\rfloor, re(i, k)\ra$.  
Let  $\{V_0,\dots,V_{a_n-1}\}$ be a partition of $V$, where $V_r=\{ \la r, s\ra : 0\leq s< k\}$  for $0\le r < a_n$, and let $\rho$ be a  $(a_n+k-2)$-colouring defined as follows: For each $r$, with $0\le r\le a_n-1$, let
\begin{enumerate}
\item The arc $(\la r,s\ra\la r+1,s\ra)$, with  $0\leq s <k$, receives color $r$.
\item If $r\geq k-2$ the arcs $(\la r,0\ra\la r,1\ra)$ and $(\la r,k-1\ra\la r+1,0\ra)$ receive color $r$; and, for  each $1\le j\le k-2$, the arc $(\la r,j\ra\la r,j+1\ra)$  receives color  $a_n-1+j$.
\item If $r\leq k-3$ the arc $(\la r,k-2-r\ra\la r,k-1-r\ra)$ receives color $r$; for each $0\leq j\leq k-3-r$ the arc $(\la r,j\ra\la r,j+1\ra)$ receives color $a_n+r+j$; for each $k-1-r\leq j\leq k-2$ the arc $(\la r,j\ra\la r,j+1\ra)$ receives color $a_n+j-(k-1-r)$; and the arc $(\la r, k-1\ra\la r+1, 0\ra)$ receives color $a_n+r$. 
\end{enumerate}

Observe that for every pair $1\leq r, r' < a_n$ the path $\la r, s\ra C_{(k)}\la r' , s\ra$ is a rainbow path with colors $r, r+1,\dots , r'-1$  (taken the sequence modulo $a_n$). Also every path $P$ of lenght at most $k-1$ in $C_{(1)}$ is rainbow. Moreover, if for some $0\leq r < a_n$, $V(P)\subseteq V_r$ then the colors appearing in $P$ are contained in $\{a_n, \dots, a_n+(k-3)\}\cup \{r\}$; and if $V(P)$ starts at $V_r$ and ends at $V_{r+1}$,  the colors of $P$ are in $\{a_n, \dots, a_n+(k-3)\}\cup \{r, r+1\}$. 

 Let $\la r,s\ra$ and $\la r',s'\ra$ be distinct vertices of $C_n(\{1,k\})$. If $r\neq r'$ it is not hard to see that either $\la r,s\ra C_{(k)}\la r',s\ra C_{(1)}\la r',s'\ra$ (if $s\leq s'$) or  $\la r,s\ra C_{(k)}\la r'-1,s\ra C_{(1)}\la r',s'\ra$ (if $s>  s'$) is a  rainbow $(\la r,s\ra\la r',s'\ra)$-path. If $r=r'$ and $s< s'$ we see that $\la r,s\ra C_{(1)}\la r,s'\ra$ is a rainbow path.  Let us suppose $r=r'$ and $s> s'$.   If no arc  $(\la r,t\ra\la r,t+1\ra)$, with $0\leq t < s'$, receives color $r$, the path $\la r-1,s\ra C_{(1)}\la r,s'\ra$ receive only colors in $\{a_n, \dots, a_n+(k-3)\}\cup \{r-1\}$, and therefore $\la r,s\ra C_{(k)}\la r-1,s\ra C_{(1)}\la r,s'\ra$ is a rainbow path.  If some arc $(\la r,t\ra\la r,t+1\ra)$, with $0\leq t < s'$, receives color $r$, by definition of $\rho$ most be either $(\la r,0\ra\la r,1\ra)$ (if $r\geq k-2$), or $(\la r,k-2-r\ra\la r,k-1-r\ra)$ (if $r\leq k-3$). For the first case in the path  $P=\la r,s\ra C_{(k)}\la a_n-1,s\ra C_{(1)}\la 0,s'\ra C_{(k)}\la r,s'\ra$,  the $k$-jumps receive colors $\{0, \dots, r, \dots, a_n-2\}$  and, by definition of $\rho$, the only $1$-jump of color $0$ is $(\la 0, k-2\ra\la 0, k-1\ra)$. Thus, since  $s'<s\leq k-1$, the colors appearing in  $\la a_n-1,s\ra C_{(1)}\la 0,s'\ra$ are contain in $\{a_n, \dots, a_n+(k-3)\}\cup \{a_n-1\}$ and therefore $P$ is rainbow.   For the second case in the path 	$P=\la r,s\ra C_{(k)}\la k-2-s,s\ra C_{(1)}\la k-1-s,s'\ra C_{(k)}\la r,s'\ra $ the $k$-jumps receive colors $\{0, \dots, k-3-s, k-1-s, \dots, a_n-1\}$ and, since $s>s'>t\geq 0$, $k-1-s\leq k-3$ and therefore the only $1$-jump of color $k-1-s$ is  $(\la k-1-s,  s-1\ra\la k-2-s, s\ra)$. Thus the colors in  $\la k-2-s,s\ra C_{(1)}\la k-1-s,s'\ra$  are contain in $\{a_n, \dots, a_n+(k-3)\}\cup \{k-2-s\}$ and $P$ is a rainbow path.  In all the cases,  from Lemma \ref{diametro}  we see that all the paths  are geodesic and the result follows.  
\end{proof}

\end{document}